\documentclass[11pt]{amsart}
\usepackage{preamble}

\begin{document}

\title[$\Gamma$-effective adjunction conjecture  with applications]{A variant of the effective adjunction conjecture with applications}


\begin{abstract}
We propose a variant of the effective adjunction conjecture for lc-trivial fibrations. This variant is suitable for inductions and can be used to treat real coefficients.
\end{abstract}

\author{Zhan Li}
\address{Department of Mathematics, Southern University of Science and Technology, 1088 Xueyuan Rd, Shenzhen 518055, China} \email{lizhan@sustech.edu.cn}

\maketitle

\tableofcontents

\section{Introduction}\label{sec: introduction}

Throughout this paper, we work with varieties defined over complex numbers.

Let $(X, D)$ be a log pair with mild singularities. Suppose that $f: X \to Z$ is a projective contraction such that $K_X+D \sim_{Z, \Qq} f^*L$ for some $\Qq$-Cartier divisor $L$ on $Z$. The canonical bundle formula (see \cite{Kaw98}) prescribes divisors $D_{\di}$ and $D_{\mo}$ on $Z$.  The divisor $D_{\di}$ is called the discriminant part which is related to the singularities of $(X, D)$, and $D_{\mo}$ is called the moduli part which is related to a polarization of the ``moduli space'' of the fibers. Moreover, we have $K_Z+D_{\di}+D_{\mo} =L$. The canonical bundle formula is pivotal in translating the study of $(X, D)$ to the lower dimensional variety $Z$.

The most mysterious part in the canonical bundle formula is the moduli part $D_{\mo}$. After certain birational modification of the base, $D_{\mo}$ is known to be nef and abundant (see \cite{Amb05}). $D_{\mo}$ is conjectured to be semi-ample, or even stronger, effectively base-point free (i.e. there exists an effective $m\in \Nn$ such that $mD_{\mo}$ is base-point free) (see \cite[Effective adjunction Conjecture 7.13.3]{PS09}). The effective adjunction conjecture is wide open with the only known case when $f$ is fibered by curves.

The purpose of this paper is to propose a variant of the effective adjunction conjecture ($\Gamma$-effective adjunction). It needs some preparations to state this conjecture and thus the precise statement is left to Section \ref{sec: Gamma effective adjunction conjecture}. We just mention that instead of predicting the moduli part to be effectively base-point free, it is predicted to be effectively $\Gamma$-base-point free. That is, there exists a finite set $\Gamma \subset (0,1]$, such that after a base change, $$K_X+D = f^*(K_Z+\ti { D}_{\di}+\ti  D_{\mo})$$ with $$\ti {D}_{\mo} = \sum r_i \ti {D}_{\mo, i}, $$ where $\sum r_i =1$ and $r_i \in \Gamma$.  There is an effective $m \in \Nn$ such that each $m{D}_{\mo, i}$ is base-point free. Besides, $(Z, \ti D_{\di})$ has the same type of singularities as $(X, D)$ in terms of discrepancies.

The $\Gamma$-effective adjunction conjecture is weaker than the effective adjunction conjecture when $(X, D)$ is klt (see Theorem \ref{thm: Q-effective adjunction implies weak effective Gamma-base point freeness}). However, $\Gamma$-effective adjunction could be used as a substitute for the effective adjunction in many applications (see Section \ref{sec: application}). The advantages of the $\Gamma$-effective adjunction lie in at least two aspects. 

(1) It is more suitable for the induction purpose. In fact, it can be put in the framework of the minimal model program. We show that the $\Gamma$-effective adjunction conjecture can be derived from the $K_X \sim_{Z, \Qq} 0$ and the relative Picard number one cases (see Proposition \ref{prop: CY and Picard 1 is enough}).

(2) It can be used to treat real coefficients. The effective adjunction only makes sense for rational coefficients.

Along the way to show Theorem \ref{thm: Q-effective adjunction implies weak effective Gamma-base point freeness}, we establish two decomposition theorems (see Theorem \ref{thm: decomposition theorem: infinite Gamma}, \ref{thm: decomposition theorem: Gamma finite}). These results are analogies of \cite{Kaw14, Nak16} for a DCC coefficient set instead of a finite set. 

Finally, we should point out that \cite{Flo14, FL19} develop inductive approaches towards the effective adjunction conjecture from other perspectives. It is an interesting question to adopt their methods in the current context.

We describe the structure of the paper. In Section \ref{sec: preliminaries}, we introduce definitions/notation and collect relevant results. The $\Gamma$-effective adjunction conjecture is stated in Section \ref{sec: Gamma effective adjunction conjecture}. Section \ref{sec: decomposition theorems} is devoted to the decomposition theorems. We study the relations between the effective adjunction conjecture and the $\Gamma$-effective adjunction conjecture in Section \ref{sec: relation between conjectures}. In Section \ref{sec: application}, we provide applications of the $\Gamma$-effective adjunction conjecture towards the boundedness problem of log canonical models and distributions of the Iitaka volumes studied in \cite{Li20a}.

\medskip

\noindent\textbf{Acknowledgements}.
This work is partially supported by a starting grant from SUSTech.

\section{Preliminaries}\label{sec: preliminaries}

\subsection{Notation and conventions}\label{subsec: Notation and conventions}
For a morphism $f: Y \to X$, we write $f: Y \to X/U$ if $f$ is a morphism over $U$. When $f$ is birational, and $B$ is a divisor on $X$, then $f_*^{-1}B$ denotes the strict transform of $B$ on $Y$. We say that $f$ is a contraction if $f_*\Oo_Y = \Oo_X$. If $f$ is a surjective morphism, then for a divisor $D$ on $Y$, we write $D^h$ (resp. $D^v$) to denote the horizontal (resp. vertical) part of $D$ over $X$. A subset $I \subset \Rr$ is called an ACC (resp. a DCC) set if it satisfying the ascending chain condition (resp. descending chain condition) with respect to the order $ \leq$. We write $D \in I$ if all the coefficients of $D$ belong to $I$.

We use $\Nn$ to denote the set of positive integers. For $k=\Zz, \Qq, \Rr$, and two divisors $A, B \in k$ on a variety $X$ over $U$, $A \sim_{U,k} B$ means that $A$ and $B$ are $k$-linearly equivalent over $U$. When $k=\Zz$ or $U=\spec \Cc$, we omit $k$ or $U$. $A \equiv B/U$ means that $A$ and $B$ are numerically equivalent over $U$

Let $X$ be a normal variety over $U$ and $B$ be an $\Rr$-divisor on $X$, then $(X, B)/U$ is called a log pair over $U$. We assume that $K_X+B$ is an $\Rr$-Cartier divisor for a log pair. Besides, $U$ is usually omitted if this is clear from the context. For a prime divisor $D$ over $X$, if $f: Y \to X$ is a birational morphism from a normal variety $Y$ such that $D$ is a divisor on $Y$, then the log discrepancy of $D$ with respect to $(X, B)$ is defined to be $a(D; X, B)\coloneqq \mult_{D}(K_Y-f^*(K_X+B))+1$. This definition is independent of the choice of $Y$. A log pair $(X, B)$ is called sub-klt (resp. sub-lc) if the log discrepancy of any divisor over $X$ is $>0$ (resp. $\geq 0$). If $B \geq 0$, then a sub-klt (resp. sub-lc) pair $(X, B)$ is called klt (resp. lc). The set of lc places of $(X, B)$ is denoted by
\begin{equation}\label{eq: LCP}
{\rm LCP}(X, B) \coloneqq \{D \mid D \text{~is a divisor over~} X \text{~with log discrepancy~} 0\}.
\end{equation}

Suppose that $(X, B)$ is a sub-lc pair and $M> 0$ is $\Rr$-Cartier. The log canonical threshold of $M$ with respect to $(X, B)$ is
\[
\lct(M;X, B) \coloneqq \sup\{t \in \Rr \mid (X, B+tM) \text{~is sub-lc}\}.
\]

Assume that $B_i, 1 \leq i \leq q$ are $\Rr$-divisors. For an $\Rr$-divisor $B$, if there exist $r_i \geq 0$ such that $B = \sum_{i=1}^q r_i B_i$ with $\sum_{i=1}^q r_i=1$, then we write
\[
B \in \Conv(B_1, \ldots, B_q).
\] In this case, we also write 
\[
(X, B) = \sum_{1 \leq i \leq q} r_i (X, B_i).
\]

Let $X$ be a normal variety. An integral b-divisor over $X$ is an element:
\[
\mathbf D \in \mathbf{Div} X = \lim_{Y \to X} \mathrm{Div} Y,
\] where the projective limit is taken over all birational models $f: Y \to X$ proper over $X$, under the push-forward homomorphism $f_*: \mathrm{Div} Y \to \mathrm{Div} X$. If $\mathbf D = \sum d_\Gamma \Gamma$ is a b-divisor on $X$, and $Y \to X$ is a birational model of $X$, then the trace of $\mathbf D$ on $Y$ is the divisor 
\[
\mathbf D_Y \coloneqq \sum_{\Gamma \text{~is a divisor on~} Y} d_\Gamma \Gamma.
\] B-divisors with coefficients in $\Qq$ or $\Rr$ are defined similarly. Let $f: Y \to X$ be a birational morphism, we write $\overline{\bf D}_Y = {\bf D}$ if the b-divisor ${\bf D}$ is obtained by the pullback from ${\bf D}_Y$. 

\begin{definition}\label{def: bpf}
For a $\Qq$-b-divisor ${\bf D}$, it is called semi-ample (resp. base-point free) if there exists a birational morphism $Y \to X$ such that $\overline{\bf D}_Y = {\bf D}$ with ${\bf D}_Y$ semi-ample (resp. base-point free). 
\end{definition}

For a b-divisor {\bf D}, the pair $(X, {\bf D})$ is called lc (resp. klt) if ${\bf D}_X \geq 0$ and $(Y, {\bf D}_Y)$ is sub-lc (resp. sub-klt) for any birational model $Y \to X$ with $Y$ a $\Qq$-factorial variety. The sheaf $\Oo_X({\bf D})$ on $X$ is defined to be $\Oo_X({\bf D}_X)$. For more on b-divisors, see \cite{Sho96, Amb04} and \cite[\S 2.3]{Cor07}. 

\subsection{Canonical bundle formula}\label{subsec: canonical bundle formula}

Suppose that $(X, B)$ is a log pair with $B$ an $\Rr$-divisor. The discrepancy b-divisor $\mathbf A = \mathbf A(X, B)$ is the $\Rr$-b-divisor of $X$ with the trace $\mathbf A_Y$ defined by the formula $K_Y =f_*(K_X+B)+\mathbf A_Y$, where $f: Y \to X$ is a proper birational morphism of normal varieties. Similarly, we define $\mathbf A^* = \mathbf A^*(X, B)$ by
\[
\mathbf A_{Y}^{*} = \sum_{a_i >-1} a_i A_i
\] for $K_Y =f^*(K_X +B)+\sum a_i A_i$. The following definition is slightly different from \cite[Definition 2.1]{Amb04} and \cite[\S 3]{FG14}.

\begin{definition}[Klt-trivial and lc-trivial fibrations]\label{def: klt-trivial fibrations}
A klt-trivial (resp. lc-trivial) fibration $f: (X,B) \to Z$ consists of a projective contraction $f: X \to Z$ between  quasi-projective normal varieties and a pair $(X,B)$ satisfying the following properties:
\begin{enumerate}
\item $(X, B)$ is sub-klt (resp. sub-lc) over the generic point of $Z$,
\item ${\rm rank} f_*\Oo_X (\lceil \mathbf A(X, B)\rceil) = 1$ (resp. ${\rm rank} f_*\Oo_X (\lceil \mathbf A^*(X, B)\rceil) = 1$), and
\item there exists an $\Rr$-Cartier $\Rr$-divisor $L$ on $Z$ such that 
\[K_X +B\sim_{\Rr} f^*L.\]
\end{enumerate}
\end{definition}

Let $f: (X, B) \to Z$ be an lc-trivial fibration and $P$ be a prime divisor on $Z$. Because $Z$ is normal, after shrinking around $P$, we can assume that $P$ is Cartier. Define
\begin{equation}\label{eq: lc at center}
\begin{split}
\lct(\eta_P; X, B)\coloneqq \max\{&t \in \Rr \mid (X, B+tf^*P) \\
&\text{~is sub-lc over the generic point of~}P\}
\end{split}
\end{equation} and set
\begin{equation}\label{eq: def of div and mod}
{D}_{\di,Z} \coloneqq \sum_{P} (1-\lct(\eta_P; X, B)) P, \quad {D}_{\mo,Z} \coloneqq L-(K_Z+{D}_{\di,Z}).
\end{equation} Then the following canonical bundle formula holds
\begin{equation}\label{eq: canonical bundle formula}
K_X+B\sim_\Qq f^*(K_Z+{D}_{\di,Z}+{D}_{\mo,Z}).
\end{equation}

In this formula, ${D}_{\di,Z}$ is called the discriminant part (or divisorial part) and ${D}_{\mo,Z}$ is called the moduli part. When $(X, B)$ is lc, there exist $\Rr$-b-divisors ${\bf D}_{\di}$ and ${\bf D}_{\mo}$ of $X$ such that the traces $(\mathbf D_{\di})_Z={D}_{\di,Z}$ and $(\mathbf D_{\mo})_Z={D}_{\mo,Z}$. Moreover, for birational morphisms $p: Z' \to Z$, $q: X' \to X$ and a morphism $f': X' \to Z'$ such that $f \circ q = p \circ f'$, we have
\begin{equation}\label{eq: b-boundary property}
q^*({K_X+B}) \sim_{\Rr} f'^*(K_{Z'}+({\bf D}_{\di})_{Z'}+({\bf D}_{\mo})_{Z'}).
\end{equation} This is shown in \cite[Theorem 0.2, Theorem 2.7]{Amb04} for $\Qq$-divisors. But the actual computation does not need this assumption (see \cite[Lemma 2.6]{Amb04}).

By inversion of adjunction (see \cite[Theorem 3.1]{Amb04}), when $(X, B)$ is klt (resp. lc), then $(Z, {\bf D}_{\di})$ is also klt (resp. lc).

 Moreover, when $K_X+B$ is $\Qq$-Cartier, then $\mathbf D_{\mo}$ is b-nef and b-abundant in the sense that there is a proper birational morphism $Z' \to Z$ and a proper surjective morphism $h: Z' \to W$ between normal varieties such that (1) $({\bf D}_{\mo})_{Z'} \sim_{\Qq} h^*H$ for some nef and big $\Qq$-divisor $H$ on $W$, and (2) ${\bf D}_{\mo} = \overline{({\bf D}_{\mo})_{Z'}}$ (i.e. ${\bf D}_{\mo}$ is the pullback of $({\bf D}_{\mo})_{Z'}$). For details, see \cite[Theorem 3.3]{Amb05} and \cite[Theorem 1.1]{FG14}.
 
 Finally, when the coefficients of $B$ belong to a DCC set $I$, then by \eqref{eq: lc at center} and Theorem \ref{thm: ACC of LCT} below, the coefficients of $({\bf D}_{\di})_Z$ belong to a DCC set $J$ which depends only on $I$ and $\dim X$.

\subsection{A collection of relevant results}\label{subsec: relevant results}

\begin{theorem}[{ACC for the log canonical threshold \cite[Theorem 1.1]{HMX14}}]\label{thm: ACC of LCT}
Fix a positive integer $d$, $I \subset (0,1]$ and a subset $J$ of the positive real numbers. If $I$ and $J$ satisfy the DCC, then 
\[
\{\lct(X, \De; M) \mid (X, \De) \text{~is lc with~} \dim X=d, \De \in I, M\in J\}
\] satisfies the ACC.
\end{theorem}

\begin{theorem}[{ACC for numerically trivial pairs \cite[Theorem D]{HMX14}}]\label{thm: ACC of num trivial}
Fix a positive integer $d$ and a set $I \subset (0,1]$, which satisfies the DCC. Then there is a finite subset $I_0\subset I$ with the following property: 

If $(X, \De)$ is a log pair such that
\begin{enumerate}
\item $X$ is projective of dimension $d$, 
\item the coefficients of $\De$ belong to $I$, 
\item  $(X, \De)$ is lc, and
\item $K_X+\De$ is numerically trivial,
\end{enumerate}
then the coefficients of $\De$ belong to $I_0$.
\end{theorem}

The following result is a generalization of \cite{Kaw14} where $X$ is assumed to be a fixed variety.

\begin{theorem}[{\cite[Theorem 1.6]{Nak16}}]\label{thm: Nak}
Fix $d\in \Nn$. Let $r_1, \ldots, r_{c'}$ be positive real numbers, and let $r_0 = 1$. Assume that $r_0, \ldots, r_{c'}$ are $\Qq$-linearly independent. Let $s_1, \ldots, s_c$: $\Rr^{c'+1} \to \Rr$ be $\Qq$-linear functions from $\Rr^{c'+1}$ to $\Rr$ (that is, the extensions of $\Qq$-linear functions from $\Qq^{c'+1}$ to $\Qq$ by taking the tensor product $\otimes_\Qq \Rr$). Assume that $s_i(r_0, \ldots, r_{c'}) \in \Rr{\geq 0}$ for each $i$. Then there exists a positive real number $\ep > 0$ such that the following holds: For any $\Qq$-Gorenstein normal variety $X$ of dimension $d$ and $\Qq$-Cartier effective Weil divisors $D_1, \ldots, D_c$ on $X$, if $(X, \sum_{1\leq i\leq c} s_i(r_0, \ldots, r_{c'})D_i)$ is lc, then $(X, \sum_{1\leq i\leq c} s_i(r_0, \ldots, r_{c'-1},t)D_i)$ is also lc for any $t$ satisfying $|t - r_{c'}| \leq \ep$.
\end{theorem}

A corollary of the above theorem and its proof is the following.

\begin{theorem}\label{thm: Nak Corollary}
Fix $d\in \Nn$, $c \in \Zz_{\geq 0}$ and $\ep>0$. Suppose that $a_1, \ldots, a_c \in \Rr_{>0}$ are positive real numbers. Let $\mathfrak S$ be the set of log pairs $(X, \sum_{1\leq i\leq c} a_iD_i)$ satisfying
\begin{enumerate}
\item $(X, \sum_{1\leq i\leq c} a_iD_i)$ is lc with $\dim X =d$, and
\item $X$ is $\Qq$-factorial, $D_i$ is a Weil divisor for each $1 \leq i \leq c$.
\end{enumerate}
Then there exist $s \in \Nn$ and $(b^{j}_1, \ldots, b^{j}_c) \in \Qq^c, 1 \leq j \leq s$ such that for any $(X, \sum_{1\leq i\leq c} a_iD_i) \in \mathfrak S$, $$(X, \sum_{1\leq i\leq c} b^j_iD_i)$$ is lc and 
\[
\sum_{1\leq i\leq c} a_iD_i \in \Conv(\sum_{1\leq i\leq c} b^1_iD_i, \dots, \sum_{1\leq i\leq c} b^s_iD_i).
\]

Moreover, $\{(b^{j}_1, \ldots, b^{j}_c) \mid 1 \leq j\leq s\}$ can be chosen such that $|a_i - b_i^j|<\ep$, and the sets of lc places (see \eqref{eq: LCP}) satisfy
\[
{\rm LCT}(X, \sum_{1\leq i\leq c} b^j_iD_i ) \subset {\rm LCT}(X, \sum_{1\leq i\leq c} a_iD_i ).
\] 
\end{theorem}
\begin{proof}
If $a_i \in \Qq$ for all $i$, then both claims automatically hold true. In the following, we assume that $\dim{\rm Span}_{\Qq}\{1, a_1, \ldots, a_c\}>1$.

Choose $\{r_0=1, r_1, \ldots,r_{c'}\}$ as a basis for ${\rm Span}_{\Qq}\{1, a_1, \ldots, a_c\}$, then there are $\Qq$-linear functions $s_i: \Rr^{c'+1} \to \Rr$ such that $a_i = s_i(r_0, \ldots, r_{c'})$. By Theorem \ref{thm: Nak}, there are $q_{c'}, p_{c'} \in \Qq$ such that $r_{c'} \in (q_{c'}, p_{c'})$ with $|r_{c'}-q_{c'}|<\ep, |r_{c'}-p_{c'}|<\ep$  and 
\[
(X, \sum_{1\leq i\leq c} s_i(r_0, \ldots, r_{c'-1},q_{c'})D_i)~, ~(X, \sum_{1\leq i\leq c} s_i(r_0, \ldots, r_{c'-1},p_{c'})D_i) 
\] are both lc. Notice that ${\rm Span}_{\Qq}\{\{1\}\cup\{s_i(r_0, \ldots, r_{c'-1},q_{c'}) \mid 1 \leq i \leq c\} \cup \{s_i(r_0, \ldots, r_{c'-1},p_{c'}) \mid 1 \leq i \leq c\} \}$ has dimension smaller than the dimension of ${\rm Span}_{\Qq}\{1, a_1, \ldots, a_c\}$. Then, an induction on dimensions proves the first claim.

The second part of the claim follows from the proof of \cite[Theorem 1.6]{Nak16}. As \cite[P. 182]{Nak16}, we consider
\[
\begin{split}
h^+ &\coloneqq \sup\{t \geq r_{c'} \mid (X, \sum_{1 \leq i \leq c} s_i(r_0, \ldots, r_{c'-1}, t) D_i) \text{~is lc}\},\\
h^- &\coloneqq \sup\{t \leq r_{c'} \mid (X, \sum_{1 \leq i \leq c} s_i(r_0, \ldots, r_{c'-1}, t) D_i) \text{~is lc}\}.
\end{split}
\] \cite[Theorem 1.6]{Nak16} shows that there exists universal $\delta>0$ such that $h^+ -r_{c'}>\delta$ and $r_{c'}-h^->\delta$. We claim that for any $h \in (r_{c'}, h^+)$ and $h' \in (h^-, r_{c'})$,
\[
\begin{split}
{\rm LCP}(X, \sum_{1 \leq i \leq c} s_i(r_0, \ldots, r_{c'-1}, h) D_i)) \subset {\rm LCP}(X, \sum_{1 \leq i \leq c} s_i(r_0, \ldots, r_{c'-1}, r_{c'}) D_i)),\\
{\rm LCP}(X, \sum_{1 \leq i \leq c} s_i(r_0, \ldots, r_{c'-1}, h') D_i)) \subset {\rm LCP}(X, \sum_{1 \leq i \leq c} s_i(r_0, \ldots, r_{c'-1}, r_{c'}) D_i)).
\end{split}
\] We only prove the first inclusion as the second one can be shown by the same argument. Suppose that $P \in {\rm LCP}(X, \sum_{1 \leq i \leq c} s_i(r_0, \ldots, r_{c'-1}, h) D_i))$. Let $\gamma \in (0,1)$ such that $h = \gamma h^++(1-\gamma) r_{c'}$. Then 
\[
\begin{split}
1&=a\big(P; X, \sum_{1 \leq i \leq c} s_i(r_0, \ldots, r_{c'-1}, h)\big)\\
&= \gamma a\big(P; X, \sum_{1 \leq i \leq c} s_i(r_0, \ldots, r_{c'-1}, h^+)\big)\\
 &\quad\quad +(1-\gamma) a\big(P; X, \sum_{1 \leq i \leq c} s_i(r_0, \ldots, r_{c'-1}, r_{c'})\big)\\
&\leq \gamma+(1-\gamma).\\
\end{split}
\] Thus $P \in {\rm LCP}(X, \sum_{1 \leq i \leq c} s_i(r_0, \ldots, r_{c'-1}, r_{c'}) D_i))$. 

We can further choose such $h, h'$ in $\Qq$ with $|r_{c'}-h|<\ep, |r_{c'}-h'|<\ep$. By induction on dimensions as in the first part, we complete the proof.
\end{proof}

We need uniform approximations for lc pairs with coefficients in a DCC set by lc pairs with coefficients in a finite set.

\begin{theorem}[{\cite[Lemma 3.2]{FM18}, \cite[Theorem 5.21]{HLS19}}]\label{thm: DCC to finite}
Let $d \in \Nn$ and $\tau \in \Rr_{>0}$ be fixed numbers. Assume that $I,I' \subset (0,1]$ are DCC sets.  Then there exists a finite set $I_0 \subset (0,1]$ depending only on $d, \tau, I$ and $I'$ satisfying the following property: 

If $(X, \De+D)$ is a log pair such that
\begin{enumerate}
\item $(X, \De+D)$ is lc with $\dim X=d$ and $X$ is a $\Qq$-factorial variety,
\item $\De, D$ do not have common components, and
\item $\De \in I', D \in I$.
\end{enumerate}

There exists $\bar D \in I_0$ such that $(X, \De+\bar D)$ is lc, $\Supp \bar D= \Supp D$, and coefficients of $\bar D - D$ belong to $[0, \tau]$.
\end{theorem}

\begin{remark}
This result is proved in \cite[Lemma 3.2]{FM18} and \cite[Theorem 5.21]{HLS19} when $\De =0$. Both arguments work without any change in the above setting. Notice that the additional assumption that ``$I\subset \Qq$ and the accumulation points of $I$ are rationals'' in \cite[Lemma 3.2]{FM18} is not needed in the argument of the above result. One can find a stronger statement in \cite[Theorem 5.21]{HLS19}. Notice that even if $I, I' \subset \Qq$, $I_0$ may not be contained in $\Qq$.
\end{remark}

The argument for the following result can be found in \cite[Proof of Theorem 3.1]{FG12}.

\begin{theorem}[{Relative abundance for numerically trivial klt pairs}]\label{thm: relatively abundance for numerically trivial klt pair}
Let $X \to Z$ be a projective contraction. Assume that $(X, B)$ is klt with $K_X+B$ a $\Qq$-Cartier divisor. If $K_X+B \equiv 0/Z$, then $K_X+B \sim_{Z, \Qq} 0$.
\end{theorem}

\begin{lemma}[{\cite[Lemma 1.7]{KMM94}}]\label{le: KMM94}
Let $f: X \to Z$ be a contraction between quasi-projective normal varieties. Let $E$ be a $\Qq$-Cartier numerically $f$ trivial divisor on $X$, where no component of $E$ dominant $Z$. Suppose that $Z$ is $\Qq$-factorial. Then $E$ is the pullback of a unique $\Qq$-divisor on $Z$.
\end{lemma}

The following result is similar to Lemma \ref{le: KMM94}. However, $Z$ is not assumed to be $\Qq$-factorial in Lemma \ref{le: rational trivial is pullback}.

\begin{lemma}\label{le: rational trivial is pullback}
Let $f: X \to Z$ be a contraction between quasi-projective normal varieties. Suppose that $E$ is an $\Rr$-Cartier divisor on $X$ such that $\Supp E$ is vertical over $Z$. If $E \sim_{Z, \Rr} 0$, then there exists a divisor $
P$ on $Z$ such that $E =f^*P$.
\end{lemma}
\begin{proof}
First, we show that it is enough to assume that $Z$ is smooth. Let $U$ be the smooth locus of $Z$, $W =f^{-1}(U)$ and $g =f|_{W}$. If there exists $P'$ on $U$ such that $E|_{W} = g^*P'$, we claim that $P = \bar P'$ (i.e. the closure of each component of $P'$) is $\Rr$-Cartier and $E = f^*P$. 

Suppose that $E  \sim_{\Rr} f^*L$, then $g^* P' \sim_{\Rr} g^*(L|_U)$ and thus $g^* (P'-L|_U) \sim_{\Rr} 0$. Let $\Theta = P'-L|_U$. First, we show that when $g^*\Theta \sim_{\Rr} 0$, then $\Theta \sim_{\Rr} 0$ on $U$. By definition, 
\[
g^*\Theta = \sum_{1 \leq i \leq q} r_i \di(h_i)
\] with $r_ i \in \Rr$ and $h_i \in K(W)$. Write $\Theta =\sum_{1 \leq j \leq q'} \theta_j \Theta_j$ with $\Theta_j$ a prime divisor and $\di (h_i)= \sum_{1 \leq k \leq q_i} a_i^k H_i^k$ with $a_i \in \Zz$, then
\[
\sum_{1 \leq j \leq q'} \theta_j g^*\Theta_j = \sum_{1 \leq i \leq q} (r_i \sum_{1 \leq k \leq q_i} a_i^k H_i^k). 
\] By comparing the coefficients of each prime divisors, we have a system of rational linear equations:
\[
f_s(\{x_j \mid 1 \leq j \leq {q'}\})= g_s(\{y_i, {z^k_i} \mid 1 \leq i \leq q, 1 \leq k \leq q_i\}), 1 \leq s \leq l,
\] where $$ {\bf w} \coloneqq (\{\theta_j\mid 1 \leq j \leq {q'}\}; \{r_i, a_i^k \mid 1 \leq i \leq q, 1 \leq k \leq q_i\})$$ is a solution. Hence, there are rational solutions ${\bf w}_t,  1 \leq t \leq w$ of the above equations such that
\[
{\bf w} = \sum_{1 \leq t \leq w} \gamma_t {\bf w}_t
\] with $\gamma_t >0$ and $\sum_{1 \leq t \leq w} \gamma_t=1$. This implies that there are $\Qq$-divisors $\Theta^t, 1 \leq t \leq w$ such that
\begin{equation}\label{eq: theta}
\Theta = \sum_{1 \leq t \leq w} \gamma_t \Theta^t
\end{equation} with $g^*\Theta^t \sim_{\Qq} 0$. Therefore, we have $m_t \in \Nn$ such that $m_t \Theta^t$ is Cartier and $\Oo_W(g^*(m_t \Theta^t)) \simeq \Oo_W$. By the projection formula, 
\[
\Oo_U(m_t \Theta^t) = g_*\Oo_W(g^*(m_t \Theta^t)) \simeq \Oo_U.
\] That is, $\Theta^t \sim_{\Qq} 0$ on $U$. By \eqref{eq: theta}, we have $\Theta \sim_{\Rr} 0$ on $U$. 

By $\Theta = P'-L|_U$, 
\[
P' \sim_{\Rr} L|_U.
\] As $\codim(X\backslash U, X) \geq 2$, $\bar P' \sim_{\Rr} L$ and thus $P=\bar P'$ is $\Rr$-Cartier. By assumption, $E|_W = g^*P'$. If we write
\[
g^*P-E = B^+- B^-,
\] where $B^+\geq 0, B^- \geq 0$ without common components, then $\codim f(B^+) \geq 2, \codim f(B^-) \geq 2$. Moreover, $B^+- B^- \equiv 0/Z$. By the negativity lemma for very exceptional divisors (for example, see \cite[Lemma 3.3]{Bir12b}), we have $B^+= B^-=0$. This shows the claim.

Now, we can assume that $Z$ is smooth. Then the argument of \cite[Lemma 1.7]{KMM94} applies to $\Rr$-divisors without any change. Or, we can use the same argument in the first part to obtain $E= \sum \gamma_i E_i$, where $E_i$ is vertical over $Z$ with $E_i \sim_{Z,\Qq} 0$,  and $\sum \gamma_i=1, \gamma_i >0$. Applying \cite[Lemma 1.7]{KMM94}  to each $E_i$, we obtain the desired result.
\end{proof}

\section{$\Gamma$-effective adjunction conjecture}\label{sec: Gamma effective adjunction conjecture}

\subsection{Effective adjunction conjecture}\label{subsec: effective adjunction} 

\cite[Conjecture 7.13]{PS09} proposes a series of conjectures concerning properties of the canonical bundle formula. We only focus on the effective adjunction conjecture in this paper.

\begin{conjecture}[{Effective adjunction, \cite[Conjecture 7.13.3]{PS09}}]\label{conj: effective adjunction}
Under the notation in Section \ref{subsec: canonical bundle formula}
, suppose that $(X, B)$ is lc with $K_X+B$ a $\Qq$-Cartier divisor. Assume that $f: (X, B) \to Z$ is an lc-trivial fibration. Then there exists a positive integer $m$ depending only on the dimension of $X$ and the horizontal multiplicities of $B$ (a finite set of rational numbers) such that $m {\mathbf D}_{\mo}$ is a base-point free b-divisor (see Definition \ref{def: bpf}).
\end{conjecture}

The statement in \cite[Conjecture 7.13.3]{PS09} is slightly restrictive by assuming additionally that there is a $\Qq$-divisor $\Theta$ on $X$ such that $K_X + \Theta\sim_{Z, \Qq} 0$ and $(F, (1- t)D|_F + t\Theta|_F)$ is klt for any $0 <t \leq 1$, where $F$ is the generic fiber of $f$.

\begin{remark}\label{rmk: fiber by curves}
This conjecture is known when general fibers of $f$ are curves (see \cite{Kod63II, Kod63III, Fuj86} and \cite[Theorem 8.1]{PS09}). \cite[Theorem 1.2]{Fuj03} establishes the semi-ample property of the moduli b-divisors when a general fiber is a K3 surface or an abelian variety.
\end{remark}

In the effective adjunction conjecture, only the coefficients of the horizontal part are required to be chosen from a fixed finite set of rationals. There is no restriction on coefficients of the vertical part (in fact, they can even be reals). The reason is given by the following result (at least in the klt case).

\begin{proposition}\label{prop: horizontal is enough}
If we assume that $(X, B)$ is klt in Conjecture \ref{conj: effective adjunction}, then for Conjecture \ref{conj: effective adjunction}, it suffices to assume that $B$ does not have vertical components over $Z$.
\end{proposition}
\begin{proof}
Let $B =B^h+B^v$, then $K_X+B^h \sim_{Z, \Rr} B^v$. Let $U= Z - f(\Supp B^v)$, then $(K_X+B^h)|_{f^{-1}(U)} \sim_{U, \Qq} 0$. By \cite[Theorem 1.1]{HX13}, $(X, B^h)$ has a good minimal model $(Y, B^h_Y)$ over $Z$. Let $B_Y^h, B_Y^v$ and $B_Y$ be the strict transforms of $B^h, B^v$ and $B$ on $Y$ respectively. By $K_X+B \sim_{Z, \Rr} 0$ and $(X, B)$ klt, $(Y, B_Y)$ is still klt. Let $p: Y \to W/Z$ be the morphism associated with $K_Y+B_Y$, then $q: W \to Z$ is a birational morphism. As $K_Y+B^h_Y+B^v_Y \sim_{Z,\Rr} 0$, we have $B^v_Y \sim_{W,\Rr} 0$. Moreover, $B^v_Y$ is vertical over $W$, hence by Lemma \ref{le: rational trivial is pullback}, there exists $P$ on $W$ such that $B^v_Y = p^*P$. Therefore,
\[
K_{Y}+B_Y= K_{Y}+B^h_Y+p^*P
\] 

Let ${\bf D}_{\di}^{B}$ (resp. ${\bf D}_{\di}^{B_Y}$, ${\bf D}_{\di}^{h}$) and ${\bf D}_{\mo}^{B}$ (resp. ${\bf D}_{\mo}^{B_Y}$, ${\bf D}_{\mo}^h$)  be the discriminant b-divisor and the moduli b-divisor for the klt-trivial fibration $(X, B) \to Z$ (resp. $(Y,B_Y) \to W$, $(Y,B^h_Y) \to W$). 

We claim:  (1)
\[
{\bf D}_{\di}^{B_Y} = {\bf D}_{\di}^h+\bar P, \text{~hence~}
{\bf D}_{\mo}^{B_Y}={\bf D}_{\mo}^h,
\] and (2)
\[
{\bf D}_{\mo}^{B} = {\bf D}_{\mo}^{B_Y}
\] viewing as b-divisors over $W$. 

For (1), let $p': Y' \to W'$ be a model of $p: Y \to W$ with birational morphisms $\tau: Y' \to Y$ and $\theta: W' \to W$ such that $p \circ \tau = \theta \circ p'$. Then
\[
\begin{split}
K_{Y'}+B_Y'=\tau^*(K_Y+B_Y) &= \tau^*(K_Y+B_Y^h) + p'^*\theta^*(P) \\
&= K_{Y'}+{(B_Y^h)'}+p'^*\theta^*(P). 
\end{split}
\] Hence for any prime divisor $\Theta$ on $W'$,
\[
\lct(\eta_\Theta; Y', {(B_Y^h)'}) =\lct(\eta_\Theta; Y', B_Y')  + \mult_\Theta(\theta^*P).
\] Therefore, (1) follows from \eqref{eq: def of div and mod}. Technically speaking, we take $L$ such that $K_Y+B^h_Y \sim_{\Rr} p^*L$, $K_Y+ B_Y\sim_{\Rr} p^*(L+P)$, and the moduli b-divisors are with respect to $L$ and $L+P$ respectively (see \eqref{eq: def of div and mod}).

For (2), by $K_X+B \sim_{Z, \Rr} 0$, for a common resolution $X \xleftarrow{s} T \xrightarrow{t} Y$, we have $K_T+B_T = s^*(K_X+B) = t^*(K_Y+B_Y)$. By the definition of the discriminant part in the canonical bundle formula, ${\bf D}_{\di}^{B} = {\bf D}_{\di}^{B_Y}$  viewing as b-divisors over $W$. Hence (2) holds true.

A combination of the above two claims shows the desired result.
\end{proof}

\subsection{$\Gamma$-effective adjunction conjecture}

Now we propose a variant of the effective adjunction conjecture. We need the following notion of $\Gamma$-base-point freeness.

\begin{definition}[$\Gamma$-base-point freeness]\label{def: Gamma-bpf}
Let $\Gamma \subset(0,1]$. An $\Rr$-b-divisor $\bf M$ is called $\Gamma$-base-point free if ${\bf M} = \sum_{i=1}^l a_i {\bf M_i}$ with $\sum_{i=1}^l a_i=1, a_i \in \Gamma$ and ${\bf M}_i$ is a base-point free b-divisor for each $i$.
\end{definition}

\begin{conjecture}[$\Gamma$-effective adjunction]\label{conj: Gamma-effective adjunction}
Suppose that $(X, D)$ is an klt (resp. lc) pair with $\dim X =d$. Let $f: (X, D) \to Z$ be an lc-trivial fibration as in Definition \ref{def: klt-trivial fibrations}. 
\begin{enumerate}
\item {\rm (Strong $\Gamma$-effective adjunction)} Assume that the coefficients of the horizontal divisors of $D$ belong to a finite set $I_h$. Then there exist a finite set $\Gamma \subset (0,1]$ and a positive integer $m$ which both depend only on $d$ and $I_h$, such that $m{\bf D}_{\mo}$ is $\Gamma$-base-point free, where ${\bf D}_{\mo}$ is the moduli b-divisor.
\item {\rm (Weak $\Gamma$-effective adjunction)} Assume that the coefficients of $D$ belong to a DCC set $I\subset (0,1]$. Then there exist a DCC set $J \subset(0,1]$, a finite set $\Gamma \subset (0,1]$ and a positive integer $m$ which all depend only on $d$ and $I$ satisfying the following property:
\begin{enumerate}
\item there is a b-divisor $\tilde {\bf D}_{\di}$, such that its trace $(\tilde {\bf D}_{\di})_{Z} \in J$,
\item there is a b-divisor $\tilde {\bf D}_{\mo}$, such that $m\tilde {\bf D}_{\mo}$ is $\Gamma$-base-point free, 
\item $(Z, \tilde {\bf D}_{\di})$ is klt (resp. lc), and
\item for birational morphisms $p: Z' \to Z$, $q: X' \to X$ and a morphism $f': X' \to Z'$ such that $f \circ q = p \circ f'$, we have
\[
q^*({K_X+D}) \sim_{\Rr} f'^*(K_{Z'}+(\tilde {\bf D}_{\di})_{Z'}+(\tilde {\bf D}_{\mo})_{Z'}).
\]
 \end{enumerate}
\end{enumerate}
\end{conjecture}

\begin{remark}
The $\tilde {\bf D}_{\di}$ and $\tilde {\bf D}_{\mo}$ in Conjecture \ref{conj: Gamma-effective adjunction} (2) may not be the discriminant and moduli b-divisors of the original fibration. 
\end{remark}

\begin{remark}
It is possible to state the ``~$\Gamma$-adjunction conjecture'' by removing ``the existence of an effective $m \in \Nn$'' in the above $\Gamma$-effective adjunction conjecture. Some of the rest results still hold true in that setting. We left appropriate modifications to interested readers.
\end{remark}

\section{Decomposition theorems}\label{sec: decomposition theorems}

The following results are generalizations of Theorem \ref{thm: Nak Corollary}.

\begin{theorem}\label{thm: decomposition theorem: infinite Gamma}
Let $d \in \Nn$ be an integer and $I \subset (0, 1]$ be a DCC set. Let 
\[
\begin{split}
\Ss \coloneqq \{(X, D) \mid & ~(X, D) \text{~lc,~} X \text{~a $\Qq$-factorial variety,~}\\
&~\dim X =d \text{~and~} D \in I\}.
\end{split}
\]
Then there exists a finite set $J \subset (0, 1] \cap \Qq$  satisfying the following property:

For any $(X, D) \in \Ss$, there exist $r_i \in \Rr_{>0}$ and divisors $D_i, 1 \leq i \leq q$  such that 
\begin{enumerate}
\item $(X, D_i)$ is lc with $D_i \in J$,
\item $\Supp D = \Supp D_i \text{~for each~} 1 \leq i \leq q$,
\item 
\[
(X, D) = \sum_{i=1}^q r_i (X, D_i) \text{~with~}\sum_{i=1}^q r_i =1.
\]
\end{enumerate} In the above, $r_i, q$ depend on the particular pair $(X, D)$.
\end{theorem}

\begin{theorem}\label{thm: decomposition theorem: Gamma finite}
Let $d \in \Nn$ be an integer.  Assume that $I \subset (0, 1]$ is a DCC set and $I' \subset (0, 1]$ is a finite set. Let 
\[
\begin{split}
\Ff \coloneqq & \{(X, \De+D) \mid (X, \De+D) \text{~lc,~} X \text{~a $\Qq$-factorial variety,~} \dim X =d,\\
&~\De \in I', D \in I\text{~and~} \De, D \text{~do not have common components}\}.
\end{split}
\]

Then there exist a finite set $\Gamma \subset (0,1]$, a DCC set $J\subset (0,1]$, and a finite set $J'  \subset (0,1] \cap \Qq$  satisfying the following property:

For any $(X, \De+D) \in \Ff$, there exist $r_i \in \Gamma$ and divisors $\De_i, D_i, 1 \leq i \leq q$  such that
\begin{enumerate}
\item $(X, \De_i+D_i)$ is lc, 
\item $\De_i \in J', D_i \in J$ and $\Supp \De_i = \Supp \De, \Supp D_i = \Supp D$,
\item 
\[
(X, \De+D) = \sum_{i=1}^q r_i (X, \De_i+D_i) \text{~with~}\sum_{i=1}^q r_i =1.
\]
\end{enumerate}

Moreover, the above statement can be strengthened in the following cases.
\begin{enumerate}[(i)]
\item If a $(X, \De+D) \in \Ff$ is klt, then $(X, \De_i+D_i)$ in (1) can be chosen to be klt as well.
\item For a $(X, \De+D) \in \Ff$, and any morphism $X \to Z$ such that $K_X+\De \equiv 0/Z$, then we can further assume that $K_X+\De_i \equiv 0/Z$ for each $i$.
\end{enumerate}
\end{theorem}

\begin{remark}
The main difference between Theorem \ref{thm: decomposition theorem: infinite Gamma} and Theorem \ref{thm: decomposition theorem: Gamma finite} is: in Theorem \ref{thm: decomposition theorem: infinite Gamma}, $J$ is a finite set while $\Gamma$ may be an infinite set. On the other hand, in Theorem \ref{thm: decomposition theorem: Gamma finite}, $J$ is a DCC set while  $\Gamma$ is a finite set, and thus $q$ is bounded above. We will use Theorem \ref{thm: decomposition theorem: Gamma finite} to study the relation between the effective adjunction conjecture and the $\Gamma$-effective adjunction conjecture.
\end{remark}

Theorem \ref{thm: decomposition theorem: infinite Gamma} and Theorem \ref{thm: decomposition theorem: Gamma finite} can be shown by similar argument.

\begin{lemma}\label{le: show for the bigger one} 
If there is a finite set $J$ such that $(X, \sum_{1\leq i \leq k} \bar d_i D_i), \bar d_i \geq d_i$ can be decomposed as in Theorem \ref{thm: decomposition theorem: infinite Gamma} by coefficients in $J$, then the same thing holds true for $(X, \sum_{1\leq i \leq k} d_i D_i)$ after enlarging $J$ (but it is still a finite set).
\end{lemma}
\begin{proof}

Fix $\delta\in\Qq_{>0}$ such that $\delta < \min J$. Replace $J$ by $J \cup \{\delta\}$. It is enough to show the claim when there exists $l$ such that $\bar d_l \geq d_l$ and $\bar d_i = d_i$ for $1 \leq i \leq k, i\neq l$.

We write $\bbeta$ for $(\beta_1, \cdots, \beta_k) \in \Rr^k$ and $\bbeta \cdot {\bf D}$ for $\sum_{i=1}^k \beta_i D_i$. If $(X, \bar d_l D_l+\sum_{1 \leq i \leq k, i\neq l} d_i D_i)$ has a decomposition $\sum_{1 \leq j \leq q} r_j (X, \bbeta_j \cdot {\bf D})$ satisfying the property claimed in Theorem \ref{thm: decomposition theorem: infinite Gamma}, then
\[
(d_1, \cdots, \bar d_l, \cdots, d_k) \in \Conv(\{\bbeta_j \mid 1 \leq j \leq q\}) 
\] with $\bbeta_j \cdot {\bf D}\in J$ and $(X, \bbeta_j \cdot {\bf D})$ lc for each $j$. Moreover, $\Supp \bbeta_j \cdot {\bf D}_j = \cup_i\Supp D_i$. 

For $\bbeta$ above, write $\balpha = (\beta_1, \cdots, \beta_{l-1}, \delta, \beta_{l+1}, \cdots, \beta_k)$.  Then
\[
(d_1, \cdots, d_l, \cdots, d_k) \in \Conv(\{\bbeta_j \mid 1 \leq j \leq q\} \cup\{\balpha_j \mid 1 \leq j \leq q\}).
\] Thus there is a decomposition of $(X, \sum_{1\leq i \leq k} d_i D_i)$ using $(X, \bbeta_i \cdot {\bf D})$ and $(X, \balpha_j \cdot {\bf D})$. Moreover, $\balpha_j \cdot {\bf D} \in J $ and $(X, \balpha_j \cdot {\bf D})$ is lc for each $j$. As $\delta>0$, we still have $\Supp \balpha_j \cdot {\bf D} = \cup_i\Supp D_i$. 
\end{proof}

\begin{proof}[Proof of Theorem \ref{thm: decomposition theorem: infinite Gamma}]
By Theorem \ref{thm: DCC to finite} (take $\De=0$), we have a finite set $I_0$, such that for a $(X, D)\in\Ss$, there exists an lc pair $(X, \bar D)$ with $\bar D \geq D$ and $\bar D \in I_0$. Consider the set of lc pairs
\[
\{(X, \bar D) \mid (X, D) \in \Ss\}.
\] By Theorem \ref{thm: Nak Corollary}, there exists a finite rational set $J$ depending only on $d$ and $I$ such that 
\[
\bar D \in \Conv(D_1, \ldots, D_q)
\] with $D_i \in J$, and $(X, D_i)$ is lc. Then by Lemma \ref{le: show for the bigger one}, we are done.
\end{proof}

\begin{proof}[Proof of Theorem \ref{thm: decomposition theorem: Gamma finite}]
We use the above notation. Replacing $I$ by $I' \cup I$, we can assume that $I' \subset I$.  Take $\tau = \frac 1 2 \min I$,  and let $I_0$ be the finite set in Theorem \ref{thm: DCC to finite}. Then for any $(X, \De+D) \in \Ff$, there exits an lc pair $(X, \De+\bar D)$ with $\bar D \in I_0$ such that $\bar D - D \in [0, \tau]$. As $I' \cup I_0$ is a finite set, by Theorem \ref{thm: Nak Corollary}, there exist a finite set $J' \subset \Qq$ and a finite set $\Gamma \subset \Rr_{>0}$ such that
\begin{equation}\label{eq: use Nak cor}
(X, \De+\bar D) = \sum_{i=1}^q r_i (X, \De_i+\bar D_i), \quad \sum_{i=1}^q r_i =1,
\end{equation} where $(X, \De_i+\bar D_i)$ is lc with $\De_i, \bar D_i \in J', r_i \in \Gamma$, and $\Supp \De_i = \Supp \De$, $\Supp \bar D_i = \Supp \bar D$. Moreover, we can assume that $\De - \De_i, \bar D - \bar D_i \in (- \frac 1 3 \tau, \frac 1 3 \tau)$. In particular, $\Supp D = \Supp \bar D = \Supp \bar D_i$, and $\bar D_i \in (\frac{5}{3}\tau, 1]$.

Thus
\begin{equation}\label{eq: decomposition}
\big(X, \De+D\big) = \sum_{i=1}^q r_i \big(X, \De_i+(\bar D_i-(\bar D -D))\big).
\end{equation} As $\bar D - D  \in [0, \tau]$, we have $\De_i+(\bar D_i-(\bar D -D)) \geq 0$ with $$\Supp \De_i = \Supp \De \text{~and~} \Supp \big(\bar D_i-(\bar D -D)\big) = \Supp D.$$ 

Because $\De_i, \bar D_i \in J'$, $\bar D \in I_0$ with both $J', I_0$ are finite sets, and  $D \in I$ which is a DCC set, coefficients of $D_i \coloneqq (\bar D_i-(\bar D -D))$ and $\De_i$ belong to a set $J$ which is still DCC. Moreover, $(X, \De_i + D_i)$ is lc as $\bar D_i \geq D_i$.

To show (i) in the second part, suppose that $(X, \De+D) \in \Ff$ is klt, and $P$ is an lc place of $(X, \De_i+D_i)$. Notice that by Theorem \ref{thm: Nak Corollary}, 
\[
{\rm LCP}(X, \De_i+\bar D_i) \subset {\rm LCP}(X, \De+\bar D)
\] for each $i$ in \eqref{eq: use Nak cor}, thus $P$ is an lc place of $(X, \De+\bar D)$. Because $a(P; X, \De+D)>0$ and
\[
\De_i+D_i = (\De_i+\bar D_i)- \big((\De+\bar D) - (\De+D) \big),
\] we have $a(P; \De_i+D_i)>0$. This is a contradiction and thus $(X, \De_i+D_i)$ is klt.

To show (ii) in the second part, let $I'=\{c_1, \ldots, c_l\}$ and the set of coefficients of all the $\bar D$ be $I_0 = \{v_1, \ldots, v_t\}$. As in Theorem \ref{thm: Nak}, assume that $c_0=1, c_1, \ldots, c_p$ are $\Qq$-linearly independent, and $\{c_0, c_1, \ldots, c_p, v_1, \ldots, v_w\}$ is a basis of $${\rm Span}_\Qq (c_0, c_1, \ldots, c_l, v_1, \ldots v_t).$$ We use $x_{c_i}$ and $x_{v_j}$ to denote the variables corresponding to $c_i$ and $v_j$ respectively.

There are rational linear functions
\[
\begin{split}
&{\bf c}_\lambda ({x_{c_i}, x_{v_j} \mid 1 \leq i \leq p, 1 \leq j \leq w}), ~ p<\lambda \leq l, \text{~and~}\\
&{\bf v}_\sigma({x_{c_i}, x_{v_j} \mid 1 \leq i \leq p, 1 \leq j \leq w}), ~ w < \sigma \leq t
\end{split}
\] such that
\[
\begin{split}
&c_\lambda = {\bf c}_\lambda({{c_i}, {v_j} \mid 1 \leq i \leq p, 1 \leq j \leq w}), ~ p<\lambda \leq l, \text{~and~}\\
&c_\sigma ={\bf v}_\sigma({{c_i}, {v_j} \mid 1 \leq i \leq p, 1 \leq j \leq w}), ~ w < \sigma \leq t.
\end{split}
\]

Suppose that $K_X+\De \equiv 0/Z$ with $\De = \sum_{i=1}^l c_i \Upxi_i$ where $\Upxi_i$ is a $\Qq$-Cartier Weil divisor (because $X$ is $\Qq$-factorial). Then there are finite rational linear equations 
\begin{equation}\label{eq: intersection eq}
{\ell}_{(X,\De)/Z, \mu}({x_{c_i}\mid 1 \leq i \leq p})=0
\end{equation} obtained by intersecting $K_X+\De$ with curve classes in $N^1(X/Z)_{\Qq}$.

Put them together, we have rational equations
\[
\begin{split}
&{\bf c}_\lambda({x_{c_i}, x_{v_j} \mid 1 \leq i \leq p, 1 \leq j \leq w})-x_{c_\lambda}=0, ~ p<\lambda \leq l\\
&{\bf v}_\sigma({x_{c_i}, x_{v_j} \mid 1 \leq i \leq p, 1 \leq j \leq w}) - x_{c_\sigma}=0, ~ w< \sigma \leq t\\
& {\ell}_{(X,\De)/Z, \mu}({x_{c_i} \mid 1 \leq i \leq p})=0 \text{~for all~} K_X+\De \equiv 0/Z.
\end{split}
\]

Although the equations may be infinite, they cut out a rational linear subspace $V\subset \Rr^{l+t+1}$. Besides, $(c_0, c_1, \ldots, c_p, v_1, \ldots, v_w) \in V$ by definition. 

By Theorem \ref{thm: Nak Corollary}, there exists a rational polytope $Q \subset \Rr^{l+t+1}$ such that $(c_0, c_1, \ldots, c_p, v_1, \ldots, v_w) \in Q$ and for any $(c'_0, c'_1, \ldots, c'_p, v'_1, \ldots, v'_w) \in Q$, 
\[
(X, \sum_{i=1}^l c_i' \Upxi_i + \sum_{j=1}^t v_j'\Theta_j)
\] is lc, where $\bar D = \sum_{j=1}^t v_j \Theta_j$ with $\Theta_j$ a $\Qq$-Cartier Weil divisor for each $j$.

Because $(c_0, c_1, \ldots, c_p, v_1, \ldots, v_w) \in V \cap Q$, $V \cap Q$ is a non-empty rational polytope. Assume that 
\[
(c_0^{(k)}, c_1^{(k)}, \ldots, c_p^{(k)}, v_1^{(k)}, \ldots, v_w^{(k)}), 1 \leq k \leq q'
\] are the vertices of $V \cap Q$. Shrinking $V \cap Q$, we can assume that $c_i^{(k)} - c_i$ and $v_j^{(k)} - v_j$ belong to $(- \frac 1 3 \tau, \frac 1 3 \tau)$ for each $i,j$ and $1 \leq k \leq q'$. 

Next, we repeat the argument in the first part. Let
\[
\De_k  = \sum_{i=1}^{l} c_i^{(k)} \Upxi_i \text{~and~} \bar D_k  = \sum_{j=1}^{t} v_j^{(k)} \Theta_j,
\] then $K_X+\De_k \equiv 0/Z$ and $(X, \De_k+\bar D_k)$ is lc. The coefficients of $\De_k$ and $\bar D_k$ belong to a finite set. Moreover, we have a finite set $\Gamma$, such that 
\[
(X, \De+\bar D) = \sum_{k=1}^{q'} r_k (X, \De_k+\bar D_k), \sum_{k=1}^{q'} r_k =1
\] with $r_k \in \Gamma$. Define $D_k \coloneqq \bar D_k - (\bar D-D)$, and let $J$ be the set of coefficients of $D_k$ and $\De_k$. Therefore, $J$ is a DCC set. Besides,  $(X, \De_i + D_i)$ is lc (klt when $(X, \De+D)$ is klt) with $\Supp \De_k= \Supp \De \text{~and~} \Supp D_k = \Supp D$. 
\end{proof}

\begin{remark}
It is possible to use decomposition theorems to study the boundedness problem and effective Iitaka fibration problem in real coefficients. For example, \cite[Theorem 1.3]{HX15} can be shown in this setting. However, there are some technical difficulties to generalize  \cite[Theorem 1.4]{HX15} to real coefficients. To avoid deviating the topic, we omit the discussions on the details.
\end{remark}

\section{Effective adjunction v.s. $\Gamma$-effective adjunction}\label{sec: relation between conjectures}

\begin{theorem}\label{thm: Q-effective adjunction implies weak effective Gamma-base point freeness}
Conjecture \ref{conj: effective adjunction} implies Conjecture \ref{conj: Gamma-effective adjunction} (2) when $(X, D)$ is klt. 
\end{theorem}

\begin{proof}[Proof of Theorem \ref{thm: Q-effective adjunction implies weak effective Gamma-base point freeness}]
We use the assumptions and notation  in Conjecture \ref{conj: Gamma-effective adjunction} (2). Taking a $\Qq$-factorial dlt modification in $(X, D)$ (for example, see \cite[Corollary 1.4.3]{BCHM10}), we can assume that $X$ is a $\Qq$-factorial variety. For a general fiber $F$ of $f$, $K_F+D^h|_F=(K_X+D)|_F \equiv 0$. By Theorem \ref{thm: ACC of num trivial}, there exists a finite set $I_h$ such that $D^h \in I_h$.  

\medskip

 \noindent Step 1. Suppose that $D^h=0$. We claim that there exists $B \in \Qq$ such that $B^h=0$, $K_X+B \sim_{Z, \Qq} 0$ and the moduli b-divisors for klt-trivial fibrations $(X, D) \to Z$ and $(X, B) \to Z$ coincide.

Let $D = \sum_{1 \leq i \leq k} d_i D_i$, then \begin{equation}\label{eq: lc is rational poly}
\mathcal V \coloneqq \{(x_1, \ldots, x_k) \in \Rr_{\geq 0}^k \mid (X,  \sum_{i=1}^k x_i D_i) \text{~is lc}\}
\end{equation} is a rational polytope. Let $[\ell_1], \ldots, [\ell_l] \in N_1(X/Z/)_{\Qq}$ be a basis, then
\[
\mathcal W \coloneqq \{(x_1, \ldots,x_k)  \in \Rr^k \mid \sum_{i=1}^k x_i (D_i\cdot \ell_{j}) = -K_{X}\cdot \ell_j \text{~for all~} \ell_j\}
\] is a rational affine subspace. By $(d_1, \ldots, d_k) \in \mathcal V \cap \mathcal W$, we see that $\mathcal V \cap \mathcal W$ is a rational polytope. By $(X, D)$ klt, there exists a klt pair $(X, B)$ such that $B \in \Qq$, $\Supp B = \Supp D$ and $K_X+B \equiv 0/Z$. By Theorem \ref{thm: relatively abundance for numerically trivial klt pair}, $K_X+B \sim_{Z, \Qq} 0$, and thus $D-B \sim_{Z, \Rr} 0$. By Lemma \ref{le: rational trivial is pullback}, there exists an $\Rr$-Cartier divisor $P$ on $Z$ such that $D-B =f^*P$. Now 
\[
K_{X}+D= K_{X}+B+f^*P.
\] 

Let ${\bf D}_{\di}^D$ (resp. ${\bf D}_{\di}^B$) and ${\bf D}_{\mo}^D$ (resp. ${\bf D}_{\mo}^B$)  be the discriminant b-divisor and the moduli b-divisor for the klt-trivial fibration $(X, D) \to Z$ (resp. $(X, B) \to Z$). As Claim (1) in the proof of Proposition \ref{prop: horizontal is enough}, we have 
\begin{equation}\label{eq: mod same}
{\bf D}_{\di}^D = {\bf D}_{\di}^B+\bar P, \text{~hence~}{\bf D}_{\mo}^B={\bf D}_{\mo}^D. 
\end{equation}

By $B \in \Qq, B^h=0$ and \eqref{eq: mod same}, applying Conjecture \ref{conj: effective adjunction} to $(X, B)$, we see that ${\bf D}_{\mo}^D$ is effectively base-point free. The rest of claims in Conjecture \ref{conj: Gamma-effective adjunction} (2) follow from the property of the canonical bundle formula (see Section \ref{subsec: canonical bundle formula}).

\medskip

 \noindent Step 2. Now, assuming that $D^h \neq 0$, we will prove the claim by induction on $\dim(X/ Z)$. Because $K_{X}+D^v\equiv -D^h/Z$ is not pseudo-effective over $Z$, we can run a $(K_{X}+D^v)$-MMP$/Z$ which terminates to a Mori fiber space $p: Y \to W$ over $Z$ (see \cite[Corollary 1.3.3]{BCHM10}). Let $q: W \to Z$ be the corresponding morphism and $w=q \circ p$. Let $D_Y$ be the strict transform of $D$ on $Y$. It is enough to show Conjecture \ref{conj: Gamma-effective adjunction} (2) for the klt-trivial fibration $w: (Y, D_Y) \to Z$. In fact, only Conjecture \ref{conj: Gamma-effective adjunction} (2)(d) needs to be justified. By $K_X+D \sim_{Z, \Rr} 0$, for a common resolution $X \xleftarrow{s} T \xrightarrow{t} Y$, we have 
 \begin{equation}\label{eq: num pullback}
 s^*(K_X+D) = t^*(K_Y+D_Y).
 \end{equation} Suppose that $\theta: Z' \to Z$, $h: X' \to X$ and $f': X' \to Z'$ with $f \circ h = \theta \circ f'$.  Assume that $h': X'' \to X'$ is a birational morphism. Then $h'^* f'^* L \sim_\Rr h'^* \Theta$ implies that $f'^* L \sim_\Rr  \Theta$. Replacing $X'$ by $X''$, we can assume that $f'$ factors through $Y'$ where $g: Y' \to Y$ is a birational morphism and $\nu: Y' \to Z'$ is a model of $w: Y \to Z$. Then Conjecture \ref{conj: Gamma-effective adjunction} (2)(d) follows by taking $T=X'$ in \eqref{eq: num pullback}.
 
 \medskip
 
  \noindent Step 3. When $\dim(Y/ W)<\dim(X/ Z)$, by the induction hypothesis (notice that the coefficients of horizontal divisors over $W$ is contained in the finite set $I^h$), there exist a finite set $\Gamma_1\subset (0,1]$, a DCC set $J'_1 \subset (0,1]$ and $m_1 \in \Nn$ depending only on $d$ and $I$, such that 
\begin{equation}\label{eq: Y/W}
{K_Y+D_Y}\sim_{\Rr}p^*({K_W}+(\ti{\bf D}_{\di})_W+(\ti{\bf D}_{\mo})_W),
\end{equation} where $\ti {\bf D}_{\di}$ and $\ti {\bf D}_{\mo}$ are b-divisors. Moreover, $(W, \ti{\bf D}_{\di})$ is klt with $(\ti{\bf D}_{\di})_W \in J'_1$ and $m_1 \ti{\bf D}_{\mo}$ is $\Gamma_1$-base-point free. We claim that there is a finite set $J_1''$ depending only on $m_1$ and $\Gamma_1$ (which in turn depending only on $d$ and $I$ ), and an effective divisor $\tilde D_{W} \sim_{\Rr} (\ti {\bf D}_{\mo})_W$ such that 
\begin{enumerate}
\item $\tilde D_{W} \in J_1''$,
\item $\Supp \tilde D_{W}$ and $\Supp (\ti{\bf D}_{\di})_W$ do not have common components,
\item $(W, (\ti{\bf D}_{\di})_W+\tilde D_{W})$ is klt.
\end{enumerate} In fact, by the definition of $\Gamma$-base-point freeness, there exists a birational morphism $\pi: W' \to W$ such that $$m_1 \ti{\bf D}_{\mo} = \sum_{1 \leq i \leq q'}  r_i {\bf M}_{i}$$ with $( {\bf M}_{i})_{W'}$ base-point free and $ {\bf M}_{i} = \overline{({\bf M}_{i})}_{W'}$. Replacing $m_1$ by $2m_1$, we can assume that $m_1 \geq 2$. Taking a general element $\ti M_{W',i} \in |({\bf M}_{i})_{W'}|$, define
\[
\ti D_W = \frac{1}{m_1}\pi_*(\sum_{1 \leq i \leq q'} r_i \ti M_{W',i})
\] and $J_1''=\{ \frac{1}{m_1} r_i \mid 1 \leq i \leq q'\}$. As $(W, \ti{\bf D}_{\di})$ is klt, $$(W', (\ti{\bf D}_{\di})_{W'}+\frac{1}{m_1}(\sum_{1 \leq i \leq q'} r_i \ti M_{W',i}))$$ is sub-klt, hence $(W, (\ti{\bf D}_{\di})_W+\tilde D_{W})$ is klt. In fact, by Conjecture \ref{conj: Gamma-effective adjunction} (2) (d), if $g: Y' \to Y$ is a birational morphism and $p': Y' \to W'$ is a morphism with $p \circ g = \pi \circ p'$, then
\begin{equation}\label{eq: pullbacks}
\begin{split}
g^*(K_Y+D_Y) &\sim_\Rr p'^*(K_{W'}+(\ti{\bf D}_{\di})_{W'}+(\ti{\bf D}_{\mo})_{W'})\\
& \sim_\Rr p'^*\pi^*(K_{W}+(\ti{\bf D}_{\di})_{W}+(\ti{\bf D}_{\mo})_{W}).
\end{split}
\end{equation} Hence 
\begin{equation}\label{eq: W'/W}
\pi^*(K_{W}+(\ti{\bf D}_{\di})_{W}+(\ti{\bf D}_{\mo})_{W}) = K_{W'}+(\ti{\bf D}_{\di})_{W'}+(\ti{\bf D}_{\mo})_{W'}.
\end{equation} Therefore,
\[
\big(W, (\ti{\bf D}_{\di})_W+\tilde D_{W}\big)=\big(W, \pi_*((\ti{\bf D}_{\di})_{W'}+\frac{1}{m_1}(\sum_{1 \leq i \leq q'} r_i \ti M_{W',i}))\big)
\] is klt. In particular, the coefficients of $(\ti{\bf D}_{\di})_W+\tilde D_{W}$ belong to a DCC set $J_1\coloneqq J'_1 \cup J''_1$. We can apply the induction hypothesis to the klt-trivial fibration $$q: (W, (\ti{\bf D}_{\di})_W+\tilde D_{W}) \to Z.$$ Notice that the coefficients of horizontal divisors over $Z$ belongs to a finite set depending only on $\dim(W/Z)$ and $J_1$, and thus depending only on $d$ and $I$. We have a DCC set $J$, a finite set $\Gamma \subset (0,1]$ and $m\in\Nn$, such that 
\[
K_W+(\ti{\bf D}_{\di})_W+\tilde D_{W}\sim_{\Rr} q^*({K_Z}+({\bf D}'_{\di})_Z+({\bf D}'_{\mo})_Z),
\] where ${\bf D}'_{\di}$ and ${\bf D}'_{\mo}$ are b-divisors such that $(\ti{\bf D}_{\di})_Z \in J$, $(Z, {\bf D}'_{\di})$ is klt and $m{\bf D}'_{\mo}$ is $\Gamma$-base-point free. To see (d) in Conjecture \ref{conj: Gamma-effective adjunction} (2), suppose that $\theta: Z' \to Z$, $g: Y' \to Y$ and $\nu: Y' \to Z'$ with $q \circ p\circ g = \theta \circ \nu$.  As in Step 2, it is enough to assume that $\nu$ factors through $W'$ with $p': Y' \to W'$ and $q': W' \to Z'$. By the construction of $\ti{\bf D}_{\di}$ and $\ti{\bf D}_{\mo}$, we have
\[
g^*(K_Y+D_Y)\sim_\Rr p'^*(K_{W'}+(\ti{\bf D}_{\di})_{W'}+(\ti{\bf D}_{\di})_{W'}). 
\] By the construction of ${\bf D}'_{\di}$ and ${\bf D}'_{\mo}$, we have
\begin{equation}\label{eq: W/Z'}
\pi^*(K_W+(\ti{\bf D}_{\di})_W+\tilde D_{W})\sim_{\Rr} q'^*({K_{Z'}}+({\bf D}'_{\di})_{Z'}+({\bf D}'_{\mo})_{Z'}),
\end{equation} where $\pi: W' \to W$. Pulling back \eqref{eq: W/Z'} through $p'$ and by \eqref{eq: pullbacks}, \eqref{eq: W'/W}, Conjecture \ref{conj: Gamma-effective adjunction} (2)(d) follows. 

Notice that $J, \Gamma, m$ depending only on $d$ and $I$. Hence $J, \Gamma, m$ and $(Z, {\bf D}'_{\di}+{\bf D}'_{\mo})$ satisfy the claim in Conjecture \ref{conj: Gamma-effective adjunction} (2).

 \medskip

 \noindent Step 4. When $\dim(Y/ W)=\dim(X/ Z)$, then $W \to Z$ is a birational morphism and thus it suffices to prove the claim for $p: Y \to W$. By Theorem \ref{thm: decomposition theorem: Gamma finite}, there is a finite set $\Gamma \subset (0,1]$, a finite rational set $J'$, and a DCC set $J$ such that
\[
(Y, D^h+ D^v) = \sum_{1 \leq i \leq q} r_i (Y, D^h_i+D^v_i)
\] with 
\begin{enumerate}
\item $(Y, D^h_i+ D^v_i)$ klt, $D_i^h \in J'$, $D^v_i \in J$,
\item $r_i \in \Gamma, \sum_{1 \leq i \leq q} r_i =1$, 
\item $\Supp D^h=\Supp D^h_i, \Supp D^v=\Supp D^v_i$. 
\item $K_Y+D_i^h \equiv 0/W$
\end{enumerate}

By $\rho(Y/W)=1$, we have $D_i^v \equiv 0/Z$. By Theorem \ref{thm: relatively abundance for numerically trivial klt pair} and a similar argument as in Step 1,
\[
K_Y+D_i^h \sim_{Z, \Qq} 0, ~ K_Y+D_i^h+D_i^v \sim_{Z, \Rr} 0.
\] Thus $D_i^v \sim_{Z, \Rr} 0$ and $p^*B_i = D_i^v$ for some $\Rr$-Cartier divisor $B_i$ on $W$ by Lemma \ref{le: rational trivial is pullback}. 

Let $D_i = D_i^h+D_i^v$. Suppose that ${\bf D}_{i,\di}$ (resp. ${\bf D}_{i,\di}^h$) and ${\bf D}_{i, \mo}$ (resp. ${\bf D}^h_{i, \mo}$) are the discriminant b-divisor and the moduli b-divisor for the klt-trivial fibration $(Y, D_i) \to W$ (resp. $(Y, D^h_i) \to W$) respectively. As Claim (1) in the proof of Proposition \ref{prop: horizontal is enough}, we have
\[
{\bf D}_{i,\di} = {\bf D}_{i,\di}^h+ \bar B_i, \text{~hence~}{\bf D}_{i, \mo}={\bf D}^h_{i, \mo}.
\]

Now applying Conjecture \ref{conj: effective adjunction} to $(Y, D_i^h) \to W$, there exists $m \in \Nn$ such that $m{\bf D}^h_{i, \mo}$ is base-point free. By the canonical bundle formula, for any birational morphisms $g: Y' \to Y$, $\pi: W' \to W$ and $p: Y \to W$, $p': Y' \to W'$ such that $p \circ g = \pi \circ p'$, we have
\[
g^*(K_Y+D_i^h) \sim_\Rr p'^*(K_{W'}+({\bf D}^h_{i,\di})_{W'}+({\bf D}^h_{i,\mo})_{W'}).
\] Thus
\[
\begin{split}
g^*(K_Y+D_i)& \sim_\Rr p'^*(K_{W'}+({\bf D}^h_{i,\di})_{W'}+\pi^*B_i+({\bf D}^h_{i,\mo})_{W'})\\
& \sim_\Rr p'^*(K_{W'}+({\bf D}_{i,\di})_{W'}+({\bf D}_{i,\mo})_{W'}).
\end{split}
\] Moreover, $(W, {\bf D}_{i,\di})$ is klt and as $D_i \in J \cup J'$ which is a DCC set, coefficients of $({\bf D}_{i,\di})_W$ still belong to a DCC set depending only on $\dim Y$ and $J \cup J'$ which in turn depending only on $d$ and $I$. Finally, by
\[
K_Y+ D = \sum_{1 \leq i \leq q} r_i (K_Y+ D_i),
\] we can take $\ti{\bf D}_{\di} \coloneqq \sum_{1 \leq i \leq q} r_i{\bf D}_{i,\di}$ and $\ti{\bf D}_{\mo} \coloneqq \sum_{1 \leq i \leq q} r_i{\bf D}_{i,\mo}$. Then $m\ti{\bf D}_{\mo}$ is $\Gamma$-base-point free, $(X, \ti{\bf D}_{\di})$ is still klt and as there are only finite possibilities for $r_i$, $(\ti{\bf D}_{\di})_W$ belongs to a DCC set depending only on $d$ and $I$. This completes the proof.
\end{proof}

\begin{remark}
The technical assumption on $(X, D)$ klt is needed to apply Theorem \ref{thm: relatively abundance for numerically trivial klt pair}. Notice that in the absolute setting, such result has already been known for numerically trivial lc pairs.
\end{remark}

\begin{remark}
When $(X, D) = \sum r_i(X, D_i)$ with $(X, D) \to Z$ and $(X, D_i) \to Z$ lc-trivial fibrations, it is not true that the moduli b-divisors satisfy ${\bf D}_{\mo} = \sum r_i {\bf D}_{i, \mo} $.
\end{remark}

According to Remark \ref{rmk: fiber by curves}, Theorem \ref{thm: Q-effective adjunction implies weak effective Gamma-base point freeness} implies the following corollary.

\begin{corollary}
Conjecture \ref{conj: Gamma-effective adjunction} (2) holds true for klt pairs when general fibers are curves.
\end{corollary}

By a similar argument as above, we show that Conjecture \ref{conj: Gamma-effective adjunction} (2) can be put in the framework of the minimal model program (MMP) and the abundance conjecture. 

\begin{proposition}\label{prop: CY and Picard 1 is enough}
Assuming the MMP and the abundance conjecture for klt pairs in relative dimensions $ \leq \dim(X/Z)$. To show Conjecture \ref{conj: Gamma-effective adjunction} (2) for klt pairs, it is enough to show the following two cases:
\begin{enumerate} 
\item $K_X \sim_{Z, \Qq} 0$, 
\item $\rho(X/Z)=1$.
\end{enumerate}
\end{proposition}
\begin{proof}
Run a $K_X$-MMP/$Z$. By assumption, it terminates to $Y/Z$, where either (i) $K_Y$ is semi-ample$/Z$, or (ii) there exists a Mori fiber space $Y \to W/Z$. Let $D_Y$ be the strict transform of $D$ on $Y$. By Step 2 in the proof of Theorem \ref{thm: Q-effective adjunction implies weak effective Gamma-base point freeness}, it suffices to show the claim for $(Y, D_Y) \to Z$. In Case (i), let $Y \to T/Z$ be the morphism induced be $K_Y$. 

In Case (i), because $K_X \equiv -D/Z$ is pseudo-effective over $Z$, the horizontal part  $D^h =0$. Thus $h: T \to Z$ is birational, and $K_Y \sim_{T, \Qq} 0$. By assumption, we have b-divisors $\ti {\bf D}_{\di}$ and $\ti {\bf D}_{\mo}$ for the klt-trivial fibration $(Y, D_Y) \to T$ (not for $Y \to T$) satisfying the claim in Conjecture \ref{conj: Gamma-effective adjunction} (2). $\ti {\bf D}_{\di}$ and $\ti {\bf D}_{\mo}$ can be viewed as b-divisors over $Z$, and it is straightforward to check (a), (b), (d) in Conjecture \ref{conj: Gamma-effective adjunction} (2). For (c), suppose that $\theta: Z' \to Z$ is a birational morphism from a $\Qq$-factorial variety and $Z' \xleftarrow{r} W \xrightarrow{s} T$ are birational morphisms such that $\theta \circ r = h \circ s$. Besides, we can assume that $m(\ti {\bf D}_{\di})_{W}$ is $\Gamma$-base-point free. Hence there exists $0 \leq M \sim_{\Rr} (\ti {\bf D}_{\mo})_{W}$ such that $(W, (\ti {\bf D}_{\di})_{W}+M)$ is sub-klt. By (d), we have
\[
\begin{split}
r^*(K_{Z'}+(\ti {\bf D}_{\di})_{Z'}+(\ti {\bf D}_{\mo})_{Z'}) &= s^*(K_T+(\ti {\bf D}_{\di})_{T}+(\ti {\bf D}_{\mo})_{T})\\
& \sim_\Rr K_W+(\ti {\bf D}_{\di})_{W}+M.
\end{split}
\] Thus $K_{Z'}+(\ti {\bf D}_{\di})_{Z'}+r_*M = r_*(K_W+(\ti {\bf D}_{\di})_{W}+M)$, and $({Z'}, (\ti {\bf D}_{\di})_{Z'}+r_*M)$ is sub-klt. By $r_*M \geq 0$, $({Z'}, (\ti {\bf D}_{\di})_{Z'})$ is still sub-klt. This shows Case (i).

\medskip

In Case (ii), apply the assumption to the Mori fiber space $Y \to W$. We have a DCC set $J_1'$, a finite set $\Gamma_1 \subset (0,1]$ and $m_1\in\Nn$ such that there exist b-divisors $\ti{\bf D}_{\di}$ and $\ti{\bf D}_{\mo}$ satisfy the claim in Conjecture \ref{conj: Gamma-effective adjunction} (2). As in the Step 3 in the proof of Theorem \ref{thm: Q-effective adjunction implies weak effective Gamma-base point freeness}, there exist a finite set $J_1''$ depending only on $\Gamma_1, m_1$, and an effective divisor $\ti D_W \sim_{\Rr} (\ti{\bf D}_{\mo})_W$ such that 
\begin{enumerate}
\item $\tilde D_{W} \in J_1''$,
\item $\Supp \tilde D_{W}$ and $\Supp (\ti{\bf D}_{\di})_W$ do not have common components,
\item $(W, (\ti{\bf D}_{\di})_W+\tilde D_{W})$ is klt.
\end{enumerate} The coefficients of $(\ti{\bf D}_{\di})_W+\tilde D_{W}$ belong to a DCC set which depends only on $\dim X$ and $I$. Moreover,
\[
(W, (\ti{\bf D}_{\di})_W+\tilde D_{W}) \to Z
\] is a klt-trivial fibration with $\dim(W/Z)<\dim(X/Z)$. Replacing $(X, D)$ by $(W, (\ti{\bf D}_{\di})_W+\tilde D_{W})$ and repeating the above argument, we obtain the desired result.
\end{proof}

We should point out that in the ``$K_X \sim_{Z,\Qq} 0$'' case, we need Conjecture \ref{conj: Gamma-effective adjunction} (2) for the pair $(X, D)$ instead of merely for $X$. The reason is the following. By $K_X+D\sim_{Z,\Rr} 0$, we have $D =f^*B$ for some $B$ on $Z$. Applying Conjecture \ref{conj: Gamma-effective adjunction} (2) to the klt-trivial fibration $X \to Z$, we have b-divisors $\ti {\bf D}_{\di}$ such that $(Z, \ti {\bf D}_{\di})$ is klt. But $(Z, \ti {\bf D}_{\di} + \bar B)$ is not necessarily klt. However, it is excepted that $\ti {\bf D}_{\di}$ is exactly the discriminant b-divisor for the klt-trivial fibration $X \to Z$. If this is the case, then $(Z, \ti {\bf D}_{\di} + \bar B)$ is still klt by the property of the canonical bundle formula.

\section{Applications of the $\Gamma$-effective adjunction}\label{sec: application}

Let $X$ be a projective variety. In \cite{Li20a}, we study the boundedness of log canonical models when the Iitaka volumes are fixed (or bounded above), and the distributions of the Iitaka volumes. Recall that the Iitaka volume (see \cite[Definition 1.1]{Li20a}) is defined to be 
\[
\Ivol(K_X+D) \coloneqq \limsup_{m\to\infty} \frac{\kappa(K_X+D)!h^0(X, \Oo_X(\lfloor m(K_X+D) \rfloor))}{m^{\kappa(K_X+D)}}
\] when $\kappa(K_X+D) \neq -\infty$ and $0$ otherwise. 

Notice that for a klt pair $(X, D)$ over $U$ with $D \in \Rr$ and the relative Kodaira dimension $\ka (K_X + D/U ) \geq 0$, $(X, D)$ admits a unique log canonical model over $U$ (see \cite{Jia20} and \cite[Corollary 1.2]{Li20b}). The following conjectures are \cite[Conjecture 1.2]{Li20a} and \cite[Conjecture 1.7]{Li20a} in terms of real coefficients.

\begin{conjecture}[{\cite[Conjecture 1.2]{Li20a}}]\label{conj: bounded of base, klt}
Let $d\in\Nn, v \in\Rr_{>0}$ be fixed numbers, and $I\subset (0,1]$ be a DCC set. Let $\Ss(d,v, I)$ be the set of varieties $Z$ satisfying the following properties:
\begin{enumerate}
\item $(X, D)$ is a projective klt pair with $\dim X =d, D\in I$,
\item $\Ivol(K_X+D)=v$, and
\item $f: X \dasharrow Z$ is the log canonical model of $(X, D)$.
\end{enumerate}
Then $\Ss(d,v,\Ii)$ is a bounded family.
\end{conjecture}

\begin{conjecture}[{\cite[Conjecture 1.6]{Li20a}}]\label{conj: DCC}
Let $d\in\Nn$ be a fixed number, and $I\subset (0,1]$ be a DCC set. Then the set of Iitaka volumes
\[
\{\Ivol(K_X+D) \mid (X, D) \text{~is a projective klt pair,~} \dim X =d, D\in I\}
\] is a DCC set.
\end{conjecture}

As \cite[Proposition 4.1]{Li20a}, we have the following result.

\begin{proposition}\label{prop: assume effective adjunction}
Assuming Conjecture \ref{conj: Gamma-effective adjunction} (2) and the existence of good minimal models, then Conjecture \ref{conj: bounded of base, klt} and Conjecture \ref{conj: DCC} hold true.
\end{proposition}
\begin{proof}
Replacing $(X, D)$ by a good minimal model, we can assume that $K_X+D$ is semi-ample with $f: X \to Z$ the morphism induced by $K_X+D$.  By Conjecture \ref{conj: Gamma-effective adjunction} (2), there exist $m \in \Nn$, a DCC set $J \in (0,1]$ and a finite set $\Gamma$ such that 
\begin{enumerate}
\item there is a b-divisor $\tilde {\bf D}_{\di}$, such that $(\tilde {\bf D}_{\di})_{Z} \in J$,
\item there is a b-divisor $\tilde {\bf D}_{\mo}$, such that $m\tilde {\bf D}_{\mo}$ is $\Gamma$-base-point free, 
\item $(Z, \tilde {\bf D}_{\di})$ is lc and $K_X+D \sim_{\Rr} f^*(K_Z+(\ti {\bf D}_{\di})_Z+(\ti {\bf D}_{\mo})_Z)$.
 \end{enumerate}

As in the Step 3 in the proof of Theorem \ref{thm: Q-effective adjunction implies weak effective Gamma-base point freeness}, there exists  an effective divisor $\ti D_Z \sim_{\Rr} (\ti{\bf D}_{\mo})_Z$ such that $(Z, (\ti{\bf D}_{\di})_Z+\tilde D_{Z})$ is klt and the coefficients of $(\ti{\bf D}_{\di})_Z+\tilde D_{Z}$ belong to a DCC set depending only on $d$ and $I$.

Notice that $K_Z+(\ti{\bf D}_{\di})_Z+\tilde D_{Z}$ is ample with $\vol(K_Z+(\ti{\bf D}_{\di})_Z+\tilde D_{Z}) = \Ivol(K_X+B)$. Then Conjecture \ref{conj: bounded of base, klt}  follows from Theorem \ref{thm: HMX18 klt} below and Conjecture \ref{conj: DCC} follows from \cite[Theorem 1.3 (1)]{HMX14}.
\end{proof}

The following result is needed in the above argument. It is essentially  \cite[Theorem 1.1]{HMX18} for real klt pairs. The argument combines \cite[Theorem 1.6]{HMX14} and \cite[Theorem 1.1]{HMX18}. For technical reasons, we need to assume the log pairs are klt instead of lc. However, the result is expected to still hold true for lc pairs.

\begin{theorem}\label{thm: HMX18 klt}
Fix an integer $d$, a positive real number $v$ and a set $I \subset (0,1]$ which satisfies the DCC. Then the set $\mathfrak F(d, v, I)$ of all log pairs $(X, \De)$ such that
\begin{enumerate}[{\rm (i)}]
\item $X$ is projective of dimension $d$,
\item $(X, \De)$ is klt,
\item the coefficients of $\De$ belong to $I$,
\item $K_X+\De$ is an ample $\Rr$-Cartier divisor, and
\item $(K_X+\De)^d=v$,
\end{enumerate}
is bounded. Besides, there is a finite set $I_0$ such that $\mathfrak F(d, v, I)=\mathfrak F(d, v, I_0)$.
\end{theorem}
\begin{proof}[Sketch of the Proof]
Following the argument of \cite[Proposition 7.3]{HMX18}, there is a projective morphism $Z \to U$ and a log smooth pair $(Z, B)$ over $U$ such that if $(X, \De) \in \mathfrak F(d, v, I)$, then there is a closed point $u \in U$ and a birational map $f_u: Z_u \dto X$ such that 
\[
\vol(Z_u, K_{Z_u}+\Phi) =v,
\] where $\Phi \leq B_u$ is the sum of the strict transform of $\De$ and the $f_u$-exceptional divisors. \cite[Lemma 2.2.2]{HMX18} implies that $f_u$ is the log canonical model of $(Z_u, \Phi)$. Besides, every stratum of $B$ has irreducible fibers over $U$ (see  \cite[Lemma 7.2]{HMX18}). Notice that the above argument works for real coefficients, rational coefficient assumption is used in  \cite[Corollary 1.3]{HMX18}. 

Now we claim that there is a finite set $I_0$ such that $\De \in I_0$. As every stratum of $B$ has irreducible fibers over $U$, we can identify components of $\De$ with components of $B$. If the claim were false, then by $B$ fixed and $\Phi \in I \cup \{1\}$, we can assume that there are two pairs $(Z_{u_i}, \Phi_i), u_i \in U, i=1,2$ such that the coefficients of $f_{1*}(\Phi_1)$ is less or equal to the corresponding coefficients of $f_{2*}(\Phi_2)$, and the inequality holds for at least one pair of coefficients. In other words, let $\Phi_{u_1}^{(2)}$ be the divisor on $Z_{u_1}$ whose coefficients are chosen as $\Phi_2$, then we have $f_{1*}(\Phi_1) <f_{1*}(\Phi_{u_1}^{(2)})$. By invariance of the plurigenera (\cite[Corollary 1.4]{HMX18}),
\begin{equation}\label{eq: vol same}
\vol(Z_{u_1}, \Phi_{u_1}^{(2)}) = \vol(Z_{u_2}, \Phi_{2})=v.
\end{equation} In fact, for $m \in \Nn$ sufficiently large, $(Z_{u_2}, \frac{\lf m\Phi_{2} \rf}{m})$ has the log canonical model $(X_i, \frac{\lf mB_2 \rf}{m})$. As $ \frac{\lf m\Phi_{2} \rf}{m} \in \Qq$, applying \cite[Corollary 1.4]{HMX18}, we have $h^0(Z_{u_1}, \frac{k\lf m\Phi_{u_1}^{(2)} \rf}{m}) = h^0(Z_{u_2}, \frac{k\lf m\Phi_{2} \rf}{m})$ for $k \in \Nn$ such that $m \mid k$. Hence, by the definition of volume,  \eqref{eq: vol same} holds true. By  \cite[Lemma 2.2.2]{HMX18}, as $\Phi_{u_1}^{(2)} > \Phi_{1}$, we see that $f_1: Z_{u_1} \dto X_1$ is the log canonical model for both $(Z_{u_1}, \Phi_1)$ and $(Z_{u_1}, \Phi_{u_1}^{(2)})$. As ${f_1}_*(\Phi_1) < {f_1}_*(\Phi_{u_1}^{(2)})$, this contradicts to $(K_{X_1}+{f_1}_*(\Phi_1))^d = (K_{X_1}+{f_1}_*(\Phi_{u_1}^{(2)}))^d=v$.

Without loss of generality, we can fix the coefficients of $\De$. Then for the corresponding $\Phi$, it can be written in two parts $\Phi=\Phi^c+\Phi^e$ where $\Phi^c$ corresponding to $\De$ and $\Phi^e$ corresponding to $f_u$-exceptional divisors. By $(X, \De)$ klt, we can choose $\ep>0$ sufficiently small such that $(Z_u, \Phi^c+(1-\ep)\Phi^e)$ still has the log canonical model $(X, \De)$. We claim that for any other $(X', \De')$ which corresponds to the closed point $u' \in U$, the log canonical model of $(Z_{u'}, {\Phi'}^c+(1-\ep) {\Phi'}^e)$ is still $(X', \De')$. Notice that by the choice of $u'$, the log canonical model of $(Z_{u'}, {\Phi'}^c+{\Phi'}^e)$ is $(X', \De')$. By $\vol(K_{Z_u}+ \Phi^c+(1-\ep)\Phi^e)=\vol(K_{Z_{u'}}+{\Phi'}^c+(1-\ep) {\Phi'}^e)=v$, we know that $(Z_{u'}, {\Phi'}^c+{\Phi'}^e)$ has some log canonical model $Z_{u'} \dto X''$ (\cite[Corollary 1.2]{Li20b}). As $$\vol(K_{Z_{u'}}+{\Phi'}^c+(1-\ep) {\Phi'}^e)=\vol(K_{Z_{u'}}+{\Phi'}^c+{\Phi'}^e)=v,$$ \cite[Lemma 2.2.2]{HMX18} implies that $Z_{u'} \dto X''$ is still the log canonical model of $(Z_{u'}, {\Phi'}^c+{\Phi'}^e)$. Hence $X''=X'$ by the uniqueness of the log canonical model (\cite[Lemma 3.6.6 (1)]{BCHM10}), and the claim is proved.

Finally, we proceed as in the proof of \cite[Theorem 1.6]{HMX14}. By \cite[Lemma 9.1]{HMX14} (this step needs klt as it uses \cite[Corollary 1.1.5]{BCHM10}), there are finite morphisms $g_i: Z \dto W_i/U$ such that for $t\in U$ and for any $0 \leq \Psi \leq {\Phi_t}^c+(1-\ep) {\Phi_t}^e$, the log canonical model of $(Z_t, \Psi)$ is $f_{it}$ for some $i$ (it seems that \cite[Lemma 9.1]{HMX14} assume $\Psi \in \Qq$ implicitly as it takes the log canonical model to be the $\proj$ of the log canonical ring. However, this also holds for $\Psi \in \Rr$ as there exists $\Psi' \in \Qq$ with $\Psi' \leq \Psi$ such that $(Z_t, \Psi')$ and $(Z_t, \Psi)$ have the same log canonical model). In particular, this shows that $(X, \De)$ belongs to a fiber of $(W_i, B_i)/U$, and thus in a bounded family.
\end{proof}

\bibliographystyle{alpha}
\bibliography{bibfile}
\end{document}